\documentclass[11pt]{article}

\usepackage[hidelinks]{hyperref}
\usepackage{amsmath,amssymb,amsthm}

\newtheorem{theorem}{Theorem}[section]
\newtheorem{lemma}[theorem]{Lemma}

\title{Liouville Theorems for Semilinear Parabolic Equations Involving the Nonlinear Gradient Term}
\author{Xue Li\textsuperscript{*}}
\date{}

\begin{document}

\maketitle

\thanks{*Corresponding author. School of Mathematical Sciences, Beijing Normal University, Beijing 100875, China. Email: xueli2024@mail.bnu.edu.cn}

\begin{abstract}
In this paper, we prove the Liouville-type theorem for equation 
$u_t-\Delta u=u^p+M|\nabla u|^q$ via the integral Bernstein method, 
where $q=2p/(p+1)$ is critical and $M<0$. As an application, 
we can establish a universal priori estimate for this equation 
by using the doubling lemma. These results generalize the previous 
work of Liang and Zhang. We can also prove the Liouville-type theorem 
for equation $u_t-\Delta u=u^p+M|\nabla u|^2$, which corresponds to the supercritical case.
\end{abstract}

\noindent\textit{Keywords:} Liouville theorem, gradient term, integral Bernstein method, universal estimate.

\noindent\textit{AMS Subject Classifications:} 35K58, 35B45, 35B53, 35B33.

\noindent\textit{Chinese Library Classifications:} O175.26.

\section{Introduction}

Liouville type theorems for elliptic and parabolic problems are important in PDE. 
The celebrated Liouville theorem for the elliptic equation $-\Delta u=u^p$ 
was proved by Gidas and Spruck in \cite{GS1981}. Later Bidaut-V\'eron and V\'eron \cite{BVV1991} 
extended and simplified such Liouville-type results on compact Riemannian manifolds. 
Chen and Li \cite{CL1991} used the moving planes method to provide a new proof of this Liouville theorem.
\begin{theorem}{\rm \cite{GS1981}}\label{Liouville theorem parabolic equation}
	The equation	
    \begin{equation}\label{elliptic equation}
		-\Delta u=u^p
	\end{equation}
    has no nontrivial nonnegative $C^2$ solution in $\mathbb{R}^N$ for $1<p<p_S$.
\end{theorem}
Here $$p_S:=
	\begin{cases}
		\frac{N+2}{N-2},\ \ &{\rm if}\ N>2,\\
		\infty,\ \ &{\rm if}\ N\leq 2.
	\end{cases}
$$
The exponent $p_S$ is critical for nonexistence. 
It is well known that when $p\geq p_S$ and $N\geq 3$ there exist positive solutions 
of (\ref{elliptic equation}) in $\mathbb{R}^N$ (see \cite{CGS1989} and \cite[Section 9]{QS2019}). 
For $p=p_S$, Gidas, Ni, and Nirenberg \cite{GNN1981} proved  all positive solutions 
of (\ref{elliptic equation}) are radially symmetric about some point under the assumption 
$u=O(|x|^{2-N})$ at infinity. Later Caffarelli, Gidas, Spruck \cite{CGS1989} and Chen, Li 
\cite{CL1991} removed the growth assumption and proved the same result.

The corresponding parabolic equation was studied initially in \cite{F1966}. 
The parabolic analogue of Theorem \ref{Liouville theorem parabolic equation} is open for a long time. 
There were many partial results, see \cite{BV1998, MS2003, PQ2006, PQS2007, Q2016, QS2019,MZ1998}. 
Finally it was solved in the full subcritical range in \cite{Q2021}.
\begin{theorem}{\rm \cite{Q2021}}
	Let $1<p<p_S$. Then the equation
	\begin{equation*}\label{parabolic equation}
		u_t-\Delta u= u^p
	\end{equation*}
	does not possess positive classical solution in $\mathbb{R}^N\times \mathbb{R}$.
\end{theorem}

Chipot and Weissler \cite{CW1989} considered the parabolic equation which involves a gradient term, namely
\begin{equation*}
	u_t-\Delta u= u^p- |\nabla u|^q.
\end{equation*} 
They established blow-up results of this equation for various ranges of the parameters $p, q$. 
Souplet \cite{S1996}  considered the blow-up results of the equation
\begin{equation*}
	u_t-\Delta u= u^p- \mu|\nabla u|^q
\end{equation*}
for small $\mu>0$ and $q=2p/(p+1)$. Numerous other researchers have investigated blow-up solutions 
of the equation, such as \cite{BVV2023, FL2020}. Recently, Liang and Zhang \cite{LZ2025} considered  
the local and global properties of nonnegative solutions for semilinear heat equation
\begin{equation}\label{main equation1}
    u_t-\Delta u=u^p+M|\nabla u|^q\ \ {\rm in}\ \Omega\times I,
\end{equation}
where $\Omega$ is a domain, $I$ is an interval of $\mathbb{R}$, and $M>0$, $p, q>1$ are parameters. 
They established the local pointwise gradient estimates when $q$ is subcritical ($q<2p/(p+1)$),
critical ($q=2p/(p+1)$) and supercritical ($q>2p/(p+1)$) with respect to $p$. With these estimates, 
they can prove the parabolic Liouville-type theorems for time-decreasing ancient solutions.

\begin{theorem}{\rm\cite{LZ2025}}\label{liouville1}
	Let $p > 1$ and $q > 2p/(p+1)$. Assume $u_t \leq 0$ in $\mathbb{R}^N \times (-\infty,0)$. 
    Then there exists a constant $c_{N,p,q} > 0$ such that for any $M > 0$, there is 
    no positive solution of {\rm (\ref{main equation1})} in $\mathbb{R}^N \times (-\infty,0)$ 
    that satisfies
	\begin{equation*}\label{eq:1.10}
		u \geq c_{N,p,q} M^{-\frac{2}{(p+1)q - 2p}}.
	\end{equation*}
\end{theorem}
And they also proved the Liouville-type theorem for the entire solutions when $1<p<p_B$ and $q$
is critical, where 
$$p_B:=\begin{cases}
    \frac{N(N+2)}{(N-1)^2},\ \ &{\rm if}\ N>1,\\
    \infty,\ \ &{\rm if}\ N=1.
\end{cases}$$
\begin{theorem}{\rm \cite{LZ2025}}\label{liouville2}
	Let $1 < p < p_B$ and $q = 2p/(p+1)$. Then there exists $\varepsilon_0(N,p) > 0$ such that for 
    any $M \in (0,\varepsilon_0)$, the nonnegative solution for {\rm (\ref{main equation1})} 
    in $\mathbb{R}^N \times \mathbb{R}$ is $u \equiv 0$.
\end{theorem}

The aim of this paper is to improve the above Liouville-type theorems. 
First we can improve Theorem \ref{liouville2} to the case $M<0$ using the integral Bernstein method. 
\begin{theorem}\label{main theorem1}
	Let $1<p<p_B$ and $q=2p/(p+1)$, then there exists $\varepsilon_1(N,p)>0$ such that for any 
    $M\in(-\varepsilon_1,\varepsilon_1)$, the nonnegative solution of {\rm (\ref{main equation1})} 
    in $\mathbb{R}^N\times \mathbb{R}$ is $u\equiv 0$.
\end{theorem}
The proof of the above theorem follows the arguments as in \cite{BV1998, S2023, LZ2025}. 
The Liouville property can lead to various universal estimates by combining the rescaling method 
with a doubling argument, see \cite{PQS12007,PQS2007,LZ2025}. Following the argument in Theorem 1.9 
from \cite{LZ2025}, we obtain the universal estimate below. 
\begin{theorem}\label{main theorem3}
	Let \( 1 < p < p_B, \ q = 2p/(p+1),\  T > 0 \) and \( \Omega \) be an arbitrary domain of 
    \( \mathbb{R}^N \). Then there exists \( \varepsilon_1 > 0 \) such that for any 
    \( M \in (-\varepsilon_1, \varepsilon_1) \), the nonnegative solution of problem 
    {\rm (\ref{main equation1})} in \( \Omega \times (0, T) \) satisfies
	\begin{equation}\label{p1}
		u(x, t) + |\nabla u(x, t)|^{\frac{2}{p+1}} \leq C(N, p) \left( t^{-\frac{1}{p-1}} 
        + (T-t)^{-\frac{1}{p-1}} + \mathrm{dist}^{-\frac{2}{p-1}}(x, \partial \Omega) \right).
	\end{equation}
\end{theorem}

Many authors extended previous results  to a broader class of nonlinearity $f(u)$, 
including exponential and  logarithmic  functions. For $f(u)=\mathrm{e}^u$, 
Herrero and Vel\'azquez\cite{HV1993} described the possible asymptotic behaviours of $u(x, t)$ 
near the blow-up point in one-dimensional case. More generally, for $f(u)=e^uL(e^u)$ with 
$L$ slowly varying at infinity, Chabi \cite{C2025} obtained the sharp global blow-up profile 
for positive radial decreasing blow-up solutions. Fujishima, Ioku, Ruf and Terraneo in 
\cite{FLRT2025} constructed singular solutions of the corresponding elliptic equation with 
exponential nonlinearities in two-dimensional case. There are few results about logarithmic functions. 
The logarithmic equation \(-\Delta u=u\log^m(1+u)(m>0)\) was recently studied in \cite{GMIQ2015}, 
while the associated parabolic equation
\(	u_t - \Delta u = |u|^{p-1} u \log^m\left(2 + u^2\right) (x\in\mathbb{R}^N, t>0, m\in\mathbb{R})\)
was studied in \cite{HZ2022}. All blow-up solutions in the Sobolev subcritical range are 
Type I solutions. This is the first determination of the blow-up rate for a semilinear heat 
equation where the main nonlinear term is not homogeneous.

In this paper we consider the positive solution of the parabolic equation
\begin{equation*}\label{parabolic log}
	u_t-\Delta u=(1+u) \ln^p(1+u).
\end{equation*}
By the Cole-Hopf type transformation \(v=\ln(1+u)\), the equation is reduced to
\begin{equation*}
	v_t-\Delta v=v^p+|\nabla v|^2.
\end{equation*}
It is formally identical to (\ref{main equation1}). But here $q=2$, it is always supercritical 
with respect to $q$ no matter what value $p$ takes. In this paper we can prove the following 
Liouville theorem for the equation 
\begin{equation}\label{main equation2}
	u_t-\Delta u=u^p+M|\nabla u|^2.
\end{equation}
 
\begin{theorem}\label{main theorem2}
	Let $1<p<p_B$, then there exists $\varepsilon_2(N,p)>0$ such that for any $M\in(0,\varepsilon_2)$, 
    the nonnegative bounded solution of {\rm(\ref{main equation2})} in $\mathbb{R}^N\times \mathbb{R}$ 
    is $u\equiv 0$.
\end{theorem}

Comparing with Theorem \ref{liouville1}, we need not assume $u$ is time-decreasing. 
And we do not require the assumption of a positive lower bound. However, we need to 
impose the assumption that $u$ possesses an upper bound.

\section{Liouville Theorem and Universal Priori Estimate for {\boldmath $M<0$}}
In this section, we prove Theorems \ref{main theorem1} and \ref{main theorem3}. First, let us recall the following lemma from \cite[Lemma 3.2]{LZ2025}, 
which provides local integral estimates for the solution and its gradient.
\begin{lemma}\label{c1}
    Let \( \Omega \) be an arbitrary domain in \( \mathbb{R}^N \), \( T > 0 \), 
    and \( 0 \leq \varphi \in C_0^\infty(\Omega \times (-T, T)) \). 
    Let \( u \) be a nonnegative solution of {\rm (\ref{main equation1})} in 
    \( \Omega \times (-T, T) \), where $M>0,\ p,q>1$ are parameters. 
    Fix \( \theta > 0 \) and set $v = u + \theta$, $v$ solves the perturbation equation
    \begin{equation}\label{w1}
        v_t-\Delta v=v^p+M|\nabla v|^q+f_\theta,
    \end{equation}
    and $f_\theta = u^p - v^p$. For given \( k \in \mathbb{R} \), denote
    \[ I = \iint \varphi v^{-2} |\nabla v|^4, \quad 
    L = \iint \varphi v^{2p}, \quad 
    G = M \iint \varphi v^{-1} |\nabla v|^{2+q},\]
    where \( \iint \) denotes the integral on \( \Omega \times (-T, T) \). Then there holds
    \begin{align}\label{b1}
        &\alpha I + \delta L - \beta G\nonumber\\
        \leq &C \iint \varphi \left( M^2 |\nabla v|^{2q} + M |v_t| |\nabla v|^q + v^{-1} 
        |v_t| |\nabla v|^2 + v^p |v_t| + M v^p |\nabla v|^q \right)\nonumber\\
        &+ C \iint \varphi v_t^2 + C \iint |\varphi_t| v^{p+1} + C \iint |\Delta \varphi| 
        |\nabla v|^2\nonumber\\
        &+ C \iint |\nabla \varphi| \left( v^{-1} |\nabla v|^3 + \bigl|v_t - v^p - M |
        \nabla v|^q\bigr| \, |\nabla v| + v^p |\nabla v| \right)\nonumber\\
        &+ C \iint \left( \varphi v^{-1} |\nabla v|^2 + \varphi v^p + |\nabla \varphi| 
        |\nabla v| \right) f_\theta + C F_\theta,
    \end{align}
    where $C = C(N, k) > 0$,
    $$F_\theta = \iint \varphi \left[ f_\theta^2 - 2 (v_t - v^p - M |\nabla v|^q) 
    f_\theta \right],$$
    and
    \begin{align*}
        &\alpha = -\frac{N-1}{N} k^2 - k,\\
        &\delta = -\frac{k}{p} \frac{N+2}{N} - \frac{N-1}{N},\\
        &\beta = \frac{N+2}{N} k.
    \end{align*}
\end{lemma} 

Since the proof of Lemma \ref{c1} is essentially independent of the sign of $M$, 
Lemma \ref{c1} can be naturally extended to the case $M<0$. 
Accordingly, the estimate \eqref{b1} can be rewritten as
\begin{align}\label{new b1}
&\alpha I + \delta L - \beta G \nonumber\\
&\leqslant C \iint \varphi \left( M^2 |\nabla v|^{2q} 
+ |M| |v_t| |\nabla v|^q + v^{-1} |v_t| |\nabla v|^2 + v^p |v_t| 
+ |M| v^p |\nabla v|^q \right) \nonumber\\
&\quad + C \iint \varphi v_t^2 + C \iint |\varphi_t| v^{p+1} 
+ C \iint |\Delta \varphi| |\nabla v|^2 \nonumber\\
&\quad + C \iint |\nabla \varphi| \left( v^{-1} |\nabla v|^3 + \bigl|v_t - v^p 
- M |\nabla v|^q\bigr| |\nabla v| 
+ v^p |\nabla v| \right) \nonumber\\
&\quad + C \iint \left( \varphi v^{-1} |\nabla v|^2 + \varphi v^p 
+ |\nabla \varphi| |\nabla v| \right) f_\theta + C F_\theta.
\end{align}

Using \eqref{new b1}, we can prove the following integral estimate 
for the solution of (\ref{main equation1}). In contrast to Lemma 3.4 in 
\cite{LZ2025}, the parameter $M$ can be negative in the present work when $q$ is critical.
\begin{lemma}\label{g1}
    Assume \( 1 < p < p_B \), $q=2p/(p+1)$. Then there exists \( \varepsilon_1 > 0 \) 
    depending on \( N \) and \( p \), such that for any \( M \in (-\varepsilon_1, 
    \varepsilon_1) \), the nonnegative solution \( u \) of {\rm(\ref{main equation1})} 
    in \( Q_R := B_R \times (-R^2, R^2) \) satisfies
    \[\int_{-R^2/2}^{R^2/2} \int_{B_{R/2}} u^{2p}(x, t) \, \mathrm{d}x\,\mathrm{d}t 
    \leq C(N, p) R^{-\frac{4p}{p-1} + N + 2}.
    \]
\end{lemma} 
\begin{proof}[\textbf{Proof}]
    First, we study the coefficients and test function. 
    Choose \( k < 0 \) and \( k \neq -1 \) such that \( \alpha, \delta > 0 \), 
    which is equivalent to
    \[\frac{p(N-1)}{N+2} < -k < \frac{N}{N-1}.\] When $N=1$, such a $k$ always exists. 
    When $N>1$, to ensure that such a $k$ exists, we need $$p < \frac{N(N+2)}{(N-1)^2},$$ 
    i.e., $p<p_B$. We estimate the right-hand side of \eqref{new b1} by choosing an appropriate test function 
    \( \varphi \). Let \( \xi \in C_0^\infty(B_1 \times (-1, 1)) \), such that \( \xi = 1 \) 
    in \( B_{1/2} \times (-1/2, 1/2) \) and \( 0 \leq \xi \leq 1 \). 
    We can fix \( \bar{\alpha} > 0 \) such that \[\frac{3p + 1}{4p} < \bar{\alpha} < 1.\]
    Fix \( R > 1 \), by taking \( \varphi(x, t) = \varphi_R(x, t) = 
    \xi^b(R^{-1}x, R^{-2}t) \) with \( b = b(\bar{\alpha}) > 2 \) sufficiently large, 
    then we can obtain 
    \begin{equation}\label{n1}
        |\nabla \varphi| \leq CR^{-1} \varphi^{\bar{\alpha}}, \quad |\Delta \varphi| 
        + \varphi^{-1}|\nabla \varphi|^2 + |\partial_t \varphi| \leq CR^{-2} 
        \varphi^{\bar{\alpha}}.   
    \end{equation}
    
    Next, we proceed to integral estimates. For convenience, we denote the integral over 
    \( Q_R \) simply by \( \iint \). Notice that $q=2p/(p+1)$, which implies
    \begin{align*}
            |\beta G|=&|\beta||M|\iint\varphi v^{-1}|\nabla v|^{2+q}\\
            \leq &\varepsilon\iint\varphi\left( v^{-\frac{2+q}{2}}|\nabla v|^{2+q}\right)^
            {\frac{4}{2+q}}+C(\varepsilon)\iint\varphi\left( |\beta||M|v^{\frac{q}{2}}\right)
            ^{\frac{4}{2-q}}\\ =&\varepsilon\iint\varphi v^{-2}|\nabla v|^4+C(\varepsilon)M
            ^{\frac{4}{2-q}}\iint\varphi v^{2p}\\
            =&\varepsilon I+C(\varepsilon)M^{\frac{4}{2-q}}L.
    \end{align*}
    Together with this, \eqref{new b1} reduces to    
    \begin{align}\label{o1}
            &\alpha I + \delta L\nonumber\\
            \leq&\,\varepsilon I+C(\varepsilon)M^{\frac{4}{2-q}}L\nonumber\\
            &+C \iint \varphi \left( M^2 |\nabla v|^{2q} + |M| |v_t| |\nabla v|^q + v^{-1} 
            |v_t| |\nabla v|^2 + v^p |v_t| + |M| v^p |\nabla v|^q \right)\nonumber\\
            &+ C \iint \varphi v_t^2 + C \iint |\varphi_t| v^{p+1} + C \iint 
            |\Delta \varphi| |\nabla v|^2\nonumber\\
            &+ C \iint |\nabla \varphi| \left( v^{-1} |\nabla v|^3 + \bigl|v_t - v^p - M |
            \nabla v|^q\bigr| \, |\nabla v| + v^p |\nabla v| \right)\nonumber\\
            &+ C \iint \left( \varphi v^{-1} |\nabla v|^2 + \varphi v^p + |\nabla \varphi| 
            |\nabla v| \right) f_\theta + C F_\theta.
    \end{align}
    Now we estimate the right side of (\ref{o1}). 
            Fix \( \varepsilon > 0 \), we first estimate the terms involving \( \varphi \). 
            By Young's inequality and \( q = 2p/(p+1) \), it follows that
    \begin{align}\label{j1}
            M^2 \varphi |\nabla v|^{2q} &\leq \varepsilon \varphi v^{2p} + C(\varepsilon) 
            M^{\frac{4}{q}} \varphi v^{\frac{2(q-2)p}{q}} |\nabla v|^4\nonumber\\
            &= \varepsilon \varphi v^{2p} + C(\varepsilon) M^{\frac{4}{q}} \varphi v^{-2} 
            |\nabla v|^4.
    \end{align}
    Meanwhile, we obtain
    \begin{align*}
        |M| \varphi |v_t| |\nabla v|^q &\leq \varphi v_t^2 + C M^2 \varphi |\nabla v|^{2q},\\
        \varphi v^{-1} |\nabla v|^2 |v_t| &\leq \frac{\alpha}{4C} \varphi v^{-2} |\nabla v|^4 
        + \frac{C}{\alpha} \varphi v_t^2,\\
        \varphi v^p |v_t| &\leq \frac{\delta}{4C} \varphi v^{2p} + \frac{C}{\delta} 
        \varphi v_t^2,\\
        |M| \varphi v^p |\nabla v|^q &\leq \varepsilon \varphi v^{2p} + C(\varepsilon) 
        \varphi M^2 |\nabla v|^{2q},
        \end{align*}
    where $C$ is the same constant as that in the first term on the right-hand side of 
    \eqref{new b1}. We next estimate the terms involving \( |\nabla \varphi| \), yielding
    \begin{align*}
            |\nabla \varphi| v^{-1} |\nabla v|^3 &\leq \varepsilon \left( 
            \varphi^{\frac{3}{4}} v^{-\frac{3}{2}} |\nabla v|^3 \right)^{\frac{4}{3}} 
            + C(\varepsilon) \left( v^{\frac{1}{2}} \varphi^{-\frac{3}{4}} 
            |\nabla \varphi| \right)^4\\
            &= \varepsilon \varphi v^{-2} |\nabla v|^4 + C(\varepsilon) v^2 \varphi^{-3} 
            |\nabla \varphi|^4\\
            &\leq \varepsilon \varphi v^{-2} |\nabla v|^4 + \varepsilon \left( 
            \varphi^{\frac{1}{p}} v^2 \right)^p + C(\varepsilon) \left( 
            \varphi^{-\frac{3p+1}{p}} |\nabla \varphi|^4 \right)^{\frac{p}{p-1}}\\
            &= \varepsilon \varphi v^{-2} |\nabla v|^4 + \varepsilon \varphi v^{2p} 
            + C(\varepsilon) \left( \varphi^{-\frac{3p+1}{4p}} |\nabla \varphi| \right)
            ^{\frac{4p}{p-1}},
    \end{align*}
    \begin{align*}
            |\nabla \varphi| |v_t| |\nabla v| &\leq \left( \varphi^{\frac{1}{2}} |v_t| 
            \right)^2 + C \left( \varphi^{-\frac{1}{2}} |\nabla \varphi| |\nabla v| 
            \right)^2\\
            &= \varphi v_t^2 + C \varphi^{-1} |\nabla \varphi|^2 |\nabla v|^2\\
            &\leq \varphi v_t^2 + \varepsilon \left( \varphi^{\frac{1}{2}}v^{-1} 
            |\nabla v|^2 \right)^2 + C(\varepsilon) \left( v |\nabla \varphi|^2 
            \varphi^{-\frac{3}{2}} \right)^2\\
            &= \varphi v_t^2 + \varepsilon \varphi v^{-2} |\nabla v|^4 + C(\varepsilon) 
            v^2 \varphi^{-3} |\nabla \varphi|^4\\
            &\leq \varphi v_t^2 + \varepsilon \varphi v^{-2} |\nabla v|^4 + \varepsilon 
            \varphi v^{2p} + C(\varepsilon) \left( \varphi^{-\frac{3p+1}{4p}} 
            |\nabla \varphi| \right)^{\frac{4p}{p-1}},
    \end{align*}
    and
    \begin{align*}
            |\nabla \varphi| v^p |\nabla v| &\leq \varepsilon \left( \varphi^{\frac{1}{4}} 
            v^{-\frac{1}{2}} |\nabla v| \right)^4 + C(\varepsilon) \left( \varphi^{-\frac{1}
            {4}} v^{p+\frac{1}{2}} |\nabla \varphi| \right)^{\frac{4}{3}}\\
            &= \varepsilon \varphi v^{-2} |\nabla v|^4 + C(\varepsilon) \varphi
            ^{-\frac{1}{3}} |\nabla \varphi|^{\frac{4}{3}} v^{\frac{4p+2}{3}}\\
            &\leq \varepsilon \varphi v^{-2} |\nabla v|^4 + \varepsilon \left( 
                \varphi^{\frac{2p+1}{3p}} v^{\frac{4p+2}{3}} \right)^{\frac{3p}{2p+1}} 
                + C(\varepsilon) \left( \varphi^{-\frac{3p+1}{3p}} |\nabla \varphi|^{\frac{4}{3}} \right)^{\frac{3p}{p-1}}\\
            &= \varepsilon \varphi v^{-2} |\nabla v|^4 + \varepsilon \varphi v^{2p} 
            + C(\varepsilon) \left( \varphi^{-\frac{3p+1}{4p}} |\nabla \varphi| \right)
            ^{\frac{4p}{p-1}}.
    \end{align*}
    Again in view of \( q = 2p/(p+1) \), it follows that
    \begin{align*}
            |M||\nabla \varphi||\nabla v|^{q+1} &\leq C |M| |\nabla \varphi||\nabla v| 
            \left[ \left( v^{-\frac{q}{2}}|\nabla v|^q \right)^{\frac{2}{q}} 
            + v^{\frac{q}{2-q}} \right]\\
            &= C |M| |\nabla \varphi| v^{-1} |\nabla v|^3 + C |M| |\nabla \varphi| v^p 
            |\nabla v|.
    \end{align*}
    We now estimate the terms involving \( |\varphi_t| \) and \( |\Delta \varphi| \). 
    Once more, via Young's inequality, we deduce
    \begin{align}\label{vp}
            |\varphi_t| v^{p+1} &\leq \varepsilon \left( \varphi^{\frac{p+1}{2p}} 
            v^{p+1} \right)^{\frac{2p}{p+1}} + C(\varepsilon) \left( \varphi^{-\frac{p+1}
            {2p}} |\varphi_t| \right)^{\frac{2p}{p-1}}\nonumber\\
            &= \varepsilon \varphi v^{2p} + C(\varepsilon) \left( \varphi^{-\frac{p+1}{2p}} 
            |\varphi_t| \right)^{\frac{2p}{p-1}}
    \end{align}
    and
    \begin{align*}
            |\Delta \varphi||\nabla v|^2 &\leq \varepsilon \left( \varphi^{\frac{1}{2}} 
            v^{-1} |\nabla v|^2 \right)^2 + C(\varepsilon) \left( \varphi^{-\frac{1}{2}} 
            v |\Delta \varphi| \right)^2\\
            &= \varepsilon \varphi v^{-2} |\nabla v|^4 + C(\varepsilon) v^2 \varphi^{-1} 
            |\Delta \varphi|^2\\
            &\leq \varepsilon \varphi v^{-2} |\nabla v|^4 + \varepsilon \left( v^2 \varphi
            ^{\frac{1}{p}} \right)^p + C(\varepsilon) \left( \varphi^{-\frac{p+1}{p}} 
            |\Delta \varphi|^2 \right)^{\frac{p}{p-1}}\\
            &= \varepsilon \varphi v^{-2} |\nabla v|^4 + \varepsilon \varphi v^{2p} 
            + C(\varepsilon) \left( \varphi^{-\frac{p+1}{2p}} |\Delta \varphi| \right)
            ^{\frac{2p}{p-1}}.
    \end{align*}
    It remains to estimate the terms involving \( f_\theta \) and \( F_\theta \). We obtain
    \begin{align*}
            \iint \varphi v^{-1} |\nabla v|^2 f_\theta &\leq \varepsilon \iint \varphi v^{-2} 
            |\nabla v|^4 + C(\varepsilon) \iint \varphi f_\theta^2,\\
            \iint \varphi v^p f_\theta &\leq \varepsilon \iint \varphi v^{2p} + C(\varepsilon) 
            \iint \varphi f_\theta^2,
    \end{align*}
    and
    \begin{align*}
            &\iint |\nabla \varphi| |\nabla v| f_\theta\\
            &\leq \varepsilon \iint \left( \varphi^{\frac{1}{4}} v^{-\frac{1}{2}} 
            |\nabla v| \right)^4 + C(\varepsilon) \iint \left( v^{\frac{1}{2}} 
            \varphi^{-\frac{1}{4}} |\nabla \varphi| f_\theta \right)^{\frac{4}{3}}\\
            &\leq \varepsilon \iint \varphi v^{-2} |\nabla v|^4 + \varepsilon \iint 
            \left( \varphi^{\frac{1}{3p}} v^{\frac{2}{3}} \right)^{3p} + C(\varepsilon) 
            \iint \left( \varphi^{-\frac{p+1}{3p}} |\nabla \varphi|^{\frac{4}{3}} f_
            \theta^{\frac{4}{3}} \right)^{\frac{3p}{3p-1}}\\
            &= \varepsilon \iint \varphi \left( v^{-2} |\nabla v|^4 + v^{2p} \right) 
            + C(\varepsilon) \iint \varphi^{-\frac{p+1}{3p-1}} |\nabla \varphi|^{\frac
            {4p}{3p-1}} f_\theta^{\frac{4p}{3p-1}}\\
            &\leq \varepsilon \iint \varphi \left( v^{-2} |\nabla v|^4 + v^{2p} \right) 
            + C(\varepsilon) \iint \left( \varphi^{-\frac{3p+1}{4p}} |\nabla \varphi| 
            \right)^{\frac{4p}{p-1}} + C(\varepsilon) \iint \varphi f_\theta^2.
    \end{align*}
    From the definition of \( F_\theta \), we obtain
    $$ F_\theta \leq \iint \varphi \left( v_t^2 + \varepsilon v^{2p} 
    + \varepsilon M^2 |\nabla v|^{2q} + C(\varepsilon) f_\theta^2 \right).$$
    Hence, inserting the above estimates into (\ref{o1}), we get
    \begin{align*}
            &\alpha I + \delta L\\
            \leq& \frac{\alpha}{4}I+\frac{\delta}{4}L+C \left( \varepsilon I 
            + \varepsilon L + C(\varepsilon) M^{\frac{4}{q}} I 
            +C(\varepsilon)M^{\frac{4}{2-q}}L+ \iint \varphi v_t^2 \right) + C(\varepsilon) 
            \iint \varphi f_\theta^2\\
            &+ C(\varepsilon) \left[ \iint \left( \varphi^{-\frac{3p+1}{4p}} 
            |\nabla \varphi| \right)^{\frac{4p}{p-1}} + \iint \left( \varphi^
            {-\frac{p+1}{2p}} |\Delta \varphi| \right)^{\frac{2p}{p-1}} 
            + \iint \left( \varphi^{-\frac{p+1}{2p}} |\varphi_t| \right)^
            {\frac{2p}{p-1}} \right].
    \end{align*}
    Equivalently,
     \begin{align}\label{i1}
            &\frac{3}{4}\alpha I + \frac{3}{4}\delta L\nonumber\\
            \leq& C \left( \varepsilon I + \varepsilon L + C(\varepsilon) M^{\frac{4}{q}} I
            +C(\varepsilon)M^{\frac{4}{2-q}}L+ \iint \varphi v_t^2 \right) + C(\varepsilon) 
            \iint \varphi f_\theta^2\nonumber\\
            &+ C(\varepsilon) \left[ \iint \left( \varphi^{-\frac{3p+1}{4p}} 
            |\nabla \varphi| \right)^{\frac{4p}{p-1}} + \iint \left( \varphi^{
            -\frac{p+1}{2p}} |\Delta \varphi| \right)^{\frac{2p}{p-1}} 
            + \iint \left( \varphi^{-\frac{p+1}{2p}} |\varphi_t| \right)
            ^{\frac{2p}{p-1}} \right].
    \end{align}
    Finally, we need to estimate the term involving the time derivative of \( v \) 
    in (\ref{i1}). From (\ref{w1}), we get
    \begin{align}\label{k1}
           & \iint \varphi v_t^2\nonumber\\
            =& \iint \varphi v_t \left( v^p + M |\nabla v|^q + \Delta v + f_\theta \right)\nonumber\\
            =& \frac{1}{p+1} \iint \varphi \left( v^{p+1} \right)_t + M \iint \varphi v_t 
            |\nabla v|^q + \iint \varphi v_t \Delta v + \iint \varphi v_t f_\theta\nonumber\\
            \leq& \frac{1}{p+1} \iint |\varphi_t| v^{p+1} + \frac{1}{2} \iint \left( \varphi 
            v_t^2 + M^2 \varphi |\nabla v|^{2q} \right)\nonumber\\
            &- \iint \left( \nabla \varphi v_t + \varphi \nabla v_t \right) \cdot \nabla v 
            + \frac{1}{4}\iint \varphi v_t^2+\iint \varphi f_\theta^2\nonumber\\
            \leq& C(p) \iint |\varphi_t| v^{p+1} + \frac{3}{4} \iint  \varphi v_t^2 + 
            \frac{1}{2}\iint M^2 \varphi |\nabla v|^{2q} - \iint \varphi \partial_t 
            \left( \frac{|\nabla v|^2}{2} \right)\nonumber\\
            &+ \varepsilon \iint \varphi v^{-2} |\nabla v|^4 + \varepsilon \iint \varphi 
            v^{2p} + C(\varepsilon) \iint \left( \varphi^{-\frac{3p+1}{4p}} |\nabla \varphi| 
            \right)^{\frac{4p}{p-1}} + \iint \varphi f_\theta^2.
    \end{align}
    Integrating by parts and using Young's inequality, we have
    \begin{align}\label{d1}
            -\iint \varphi \partial_t \left( \frac{|\nabla v|^2}{2} \right) &= 
            \frac{1}{2} \iint \varphi_t |\nabla v|^2\nonumber\\
            &\leq \varepsilon \iint \varphi v^{-2} |\nabla v|^4 + \varepsilon 
            \iint \varphi v^{2p} + C(\varepsilon) \iint \left( \varphi^{-\frac{p+1}{2p}} 
            |\varphi_t| \right)^{\frac{2p}{p-1}}.
    \end{align}
    Inserting (\ref{j1}), (\ref{vp}) and (\ref{d1}) into (\ref{k1}) yields
    \begin{align}\label{l1}
            \iint \varphi v_t^2 \leq &C \left( \varepsilon I + \varepsilon L 
            + C(\varepsilon) M^{\frac{4}{q}} I \right)\nonumber\\
            &+ C(\varepsilon) \iint \left[ \left( \varphi^{-\frac{3p+1}{4p}} 
            |\nabla \varphi| \right)^{\frac{4p}{p-1}} + \left( \varphi^{-\frac{p+1}{2p}} 
            |\varphi_t| \right)^{\frac{2p}{p-1}} \right] + C \iint \varphi f_\theta^2.
    \end{align}
    Combining (\ref{l1}) and (\ref{i1}), we obtain
    \begin{align*}
            &\frac{3}{4}\alpha I +\frac{3}{4} \delta L\\
            \leq &C \left( \varepsilon I + \varepsilon L + C(\varepsilon) M^{\frac{4}{q}} I 
            +C(\varepsilon)M^{\frac{4}{2-q}}L \right) + C(\varepsilon) \iint 
            \varphi f_\theta^2\\
            &+ C(\varepsilon) \iint \left[ \left( \varphi^{-\frac{3p+1}{4p}} 
            |\nabla \varphi| \right)^{\frac{4p}{p-1}} + \left( \varphi^{-\frac{p+1}{2p}} 
            |\Delta \varphi| \right)^{\frac{2p}{p-1}} + \left( \varphi^{-\frac{p+1}{2p}} 
            |\varphi_t| \right)^{\frac{2p}{p-1}} \right].
    \end{align*}
    We select \( \varepsilon = \varepsilon(N, p) > 0 \) sufficiently small such that 
    $C\varepsilon<\min\{\frac{\alpha}{4},\frac{\delta}{4}\}$. And there exists 
    \( \varepsilon_1(N, p) > 0 \) such that, for \( -\varepsilon_1 < M < \varepsilon_1 \),
    \[ C(\varepsilon) M^{\frac{4}{q}} = C(N, p) M^{\frac{4}{q}} \leq \frac{\alpha}{4}.\]
    \[C(\varepsilon) M^{\frac{4}{2-q}} = C(N, p) M^{\frac{4}{2-q}} \leq \frac{\delta}{4}.\]
    Together with (\ref{n1}), and using the fact \( \bar{\alpha} > (3p+1)/(4p) \), we obtain
    \[\frac{\alpha}{4} I + \frac{\delta}{4} L \leq C(N, p) R^{-\frac{4p}{p-1} + N + 2} 
    + C(\varepsilon) \iint \varphi f_\theta^2.\]
    Using the fact that 
    \[\int_{-R^2/2}^{R^2/2} \int_{B_{R/2}} \varphi v^{2p}(x, t) \, {\rm d}x \, {\rm d}t 
    \to \int_{-R^2/2}^{R^2/2} \int_{B_{R/2}} u^{2p}(x, t) \, {\rm d}x \, {\rm d}t \quad 
    \text{as } \theta \to 0,\]
    then Lemma \ref{g1} follows as $\theta\rightarrow 0$.
\end{proof}

\begin{proof}[\textbf{The proof of Theorem \ref{main theorem1}.}]
            According to Lemma \ref{g1},
            \[\int_{-R^2/2}^{R^2/2} \int_{B_{R/2}} u^{2p}(x, t) \, {\rm d}x \, {\rm d}t 
            \leq C(N, p) R^{-\frac{4p}{p-1} + N+2}.\]
            Since \( p < p_B \leq p_S \), letting \( R \to \infty \), we conclude that 
            \( u \equiv 0 \).
\end{proof}
    
To establish Theorem\ref{main theorem3}, we first recall the following key doubling lemma, 
which is due to Pol\'{a}\v{c}ik, Quittner and Souplet \cite[Lemma 5.1]{PQS12007}.
\begin{lemma}\label{lem3}
    Let \( (\mathcal{X}, d) \) be a complete metric space with metric \( d \), 
    and let \( \emptyset \neq D \subset \Sigma \subset \mathcal{X} \), 
    with \( \Sigma \) closed. Set \( \Gamma = \Sigma \setminus D \). 
    Let \( M : D \to (0, \infty) \) be bounded on compact subsets of \( D \) 
    and fix a real number \( k > 0 \). If \( y \in D \) is such that
    $$M(y) \, {\rm dist}(y, \Gamma) > 2k,$$
    then there exists \( x \in D \) such that
    $$M(x) \, {\rm dist}(x, \Gamma) > 2k, \quad M(x) \geq M(y),$$
    and
    $$M(z) \leq 2M(x) \quad \text{for all } z \in D \cap \overline{B}_{\mathcal{X}}(x, kM^{-1}(x)).$$
\end{lemma}
\begin{proof}[\textbf{The proof of Theorem \ref{main theorem3}.}]
    Using the parabolic distance
    $$d_P((x, t), (y, s)) := |x - y| + |t - s|^{\frac{1}{2}},$$ 
    the estimates \eqref{p1} is written as a more concise form in \cite{PQS2007}. 
    Namely, denote \( D = \Omega \times (0, T) \), with the topological boundary 
    \(\partial D = (\partial \Omega \times [0, T]) \cup (\overline{\Omega} \times \{0, T\}) \) 
    in \( \mathbb{R}^{N+1} \), and \( d_P((x, t), \partial D) = \inf_{(y,s) 
    \in \partial D} d_P((x, t), (y, s)) \). We will prove the more precise estimates
        \begin{equation}\label{q1}
            u(x, t) + |\nabla u(x, t)|^{\frac{2}{p+1}} \leq C d_P^{-\frac{2}{p-1}}((x, t), \partial D) 
            \quad \text{in } D
        \end{equation}
    instead of (\ref{p1}).
    Set$$M(u) = u^{\frac{p-1}{2}} + |\nabla u|^{\frac{p-1}{p+1}},$$
    then (\ref{q1}) can be written as
        \begin{equation*}
            M(u(x, t)) \leq C d_P^{-1}((x, t), \partial D), \quad (x, t) \in D.
        \end{equation*}
    Assume this estimate fails, then there exist sequences of domains \( \Omega_k \),  
    times \( T_k > 0 \), solutions \( u_k \) of (\ref{main equation1}) in \( D_k := \Omega_k 
    \times (0, T_k) \) and points \( (y_k, \tau_k) \in D_k \) such that
        \begin{equation*}
            M_k(y_k, \tau_k) := M(u_k(y_k, \tau_k)) > 2kd_P^{-1}((y_k, \tau_k), \partial D_k).
        \end{equation*}
    We use Lemma \ref{lem3} with \( \mathcal{X} = \mathbb{R}^{N+1} \), equipped with 
    the parabolic distance \( d_P \), \( \Sigma = \Sigma_k = \overline{D}_k \) and \( \Gamma = 
    \partial D_k \). Then there exist points \( (x_k, t_k) \in D_k \), such that
        \begin{equation}\label{r1}
            M_k(x_k, t_k) > 2k d_P^{-1}((x_k, t_k), \partial D_k), \quad M_k(x_k, t_k) 
            \geq M_k(y_k, \tau_k), 
        \end{equation}
    and
    \begin{equation}\label{u1}
        M_k(x, t) \leq 2M_k(x_k, t_k) \quad \text{in } \hat{D}_k,
    \end{equation}
    where $\hat{D}_k := \left\{(x,t)\in \mathbb{R}^{N+1}: d_P((x, t), (x_k, t_k))
    \leq kM_k^{-1}(x_k, t_k) \right\}.$
    By the first inequality in (\ref{r1}), we obtain
    $$d_P((x, t), (x_k, t_k)) < \frac{1}{2}d_P((x_k, t_k), \partial D_k), \quad (x, t) 
    \in \hat{D}_k,$$
    which implies \( \hat{D}_k \subset D_k \). Set
    \begin{equation}\label{s1}
        \lambda_k = M_k^{-1}(x_k, t_k),
    \end{equation}
    and rescale \( u_k \) by setting
    \begin{equation}\label{t1}
        v_k(y, s) = \lambda_k^{\frac{2}{p-1}} u_k \left( x_k + \lambda_k y, t_k 
        + \lambda_k^2 s \right), 
        \quad (y, s) \in \tilde{D}_k,
    \end{equation}
    where
    $$\tilde{D}_k := \left\{ y \in \mathbb{R}^N: \ |y| < \frac{k}{2} \right\} \times 
    \left( -\frac{k^2}{4}, \frac{k^2}{4} \right).$$
    The function \( v_k \) solves
    $$\partial_s v_k - \Delta v_k = v_k^p + M|\nabla v_k|^{\frac{2p}{p+1}} \quad 
    \text{in } \tilde{D}_k.$$
    By (\ref{s1}) and (\ref{t1}), we have
    $$\left[ v_k^{\frac{p-1}{2}} + |\nabla v_k|^{\frac{p-1}{p+1}} \right](0, 0) 
    = \lambda_k M_k(x_k, t_k) = 1.$$
    By (\ref{u1}), we also have
    $$\left[ v_k^{\frac{p-1}{2}} + |\nabla v_k|^{\frac{p-1}{p+1}} \right](y, s) \leq 2, 
    \quad (y, s) \in \tilde{D}_k.$$
    Using the parabolic \( L^p \) estimates and embedding theorem, up to a subsequence, 
    there exists a \( v \in W_{\text{loc}}^{2,1;r}(\mathbb{R}^N\times\mathbb{R}) \) such that
    $$v_k \to v \quad \text{in } C_{\text{loc}}^{2,1}(\mathbb{R}^N \times \mathbb{R})$$
    for any \( 1 < r < \infty \). Notice that $\left[ v_k^{(p-1)/2} + |\nabla v_k|^{(p-1)/(p+1)} 
    \right](0, 0) = 1$, hence \( v \) is nontrivial.
    Therefore, under the assumptions of Theorem \ref{main theorem3}, 
    \( v \) is a nonnegative nontrivial solution of equation (\ref{main equation1}) in 
    \( \mathbb{R}^N \times \mathbb{R} \) with \( -\varepsilon_1 < M < \varepsilon_1 \). 
    This contradicts Theorem \ref{main theorem1}, and thus Theorem \ref{main theorem3} is proved.
\end{proof}

\section{Liouville Theorem for {\boldmath $q=2$}}
In this section, we prove Theorem \ref{main theorem2}.
\begin{proof}[\textbf{The proof of Theorem \ref{main theorem2}.}]
    Since the choice of coefficients and test function is the same as that in Lemma \ref{g1}, 
    we proceed directly to the integral estimates.
            
    Choose $q=2$ in Lemma \ref{c1}, (\ref{b1}) becomes the following form
    \begin{align}\label{v1}
        &\alpha I + \delta L - \beta G\nonumber\\
        \leq &C \iint \varphi \left( M^2 |\nabla v|^4 + M |v_t| |\nabla v|^2 + v^{-1} |v_t| |\nabla v|^2 + v^p |v_t| + M v^p |\nabla v|^2 \right)\nonumber\\
        &+ C \iint \varphi v_t^2 + C \iint |\varphi_t| v^{p+1} + C \iint |\Delta \varphi| |\nabla v|^2\nonumber\\
        &+ C \iint |\nabla \varphi| \left( v^{-1} |\nabla v|^3 + \bigl|v_t - v^p - M |\nabla v|^2\bigr| \, |\nabla v| + v^p |\nabla v| \right)\nonumber\\
        &+ C \iint \left( \varphi v^{-1} |\nabla v|^2 + \varphi v^p + |\nabla \varphi| |\nabla v| \right) f_\theta + C F_\theta,
    \end{align}
    where $C = C(N, k) > 0$,\ $v$ satisfies
    \begin{equation}\label{eq2}
             v_t-\Delta v=v^p+M|\nabla v|^2+f_\theta,
    \end{equation}
    and $$F_\theta = \iint \varphi \left[ f_\theta^2 - 2 (v_t - v^p - M |\nabla v|^2) f_\theta \right].$$
    Now we estimate the right side of (\ref{v1}) as follows. Since the integral estimates  
    are essentially the same as those in the proof of Lemma \ref{g1}, 
    the estimates for the same terms are omitted, and we only provide details 
    for the terms that are different.  

    Fix \( \varepsilon > 0 \), we first estimate the terms involving \( \varphi \). Denote
    $$E=\iint \varphi|\nabla v|^4,$$
    then
    \begin{equation}\label{me}
            \iint M^2 \varphi |\nabla v|^4=M^2 E.
    \end{equation}
    Using Young's inequality, we obtain
    $$M\varphi |v_t||\nabla v|^2\leq \varphi v_t^2+CM^2\varphi |\nabla v|^4,$$
    and
    $$M\varphi v^p|\nabla v|^2\leq \varepsilon\varphi v^{2p}+C(\varepsilon)M^2\varphi|\nabla v|^4.$$
    We next estimate the terms involving \( |\nabla \varphi| \). By the boundedness assumption on $u$, 
    we have $0<v<\tilde{C}$. Hence 
    \begin{align*}
             M|\nabla\varphi||\nabla v|^3&\leq\varepsilon\left(\varphi^{\frac{3}{4}}v^{-\frac{3}{2}}|\nabla v|^3\right)^{\frac{4}{3}}+C(\varepsilon)\left(M\varphi^{-\frac{3}{4}}|\nabla\varphi|v^{\frac{3}{2}}\right)^4\\
             &=\varepsilon\varphi v^{-2}|\nabla v|^4+C(\varepsilon)M^4 v^6\varphi^{-3}|\nabla \varphi|^4 \\
             &\leq\varepsilon\varphi v^{-2}|\nabla v|^4+ C(\varepsilon)\tilde{C}^4 M^4 v^2\varphi^{-3}|\nabla \varphi|^4 \\
             &\leq\varepsilon\varphi v^{-2}|\nabla v|^4+\varepsilon\left(\varphi^{\frac{1}{p}}v^2\right)^p+C(\varepsilon)\left(\tilde{C}^4 M^4 \varphi^{-\frac{3p+1}{p}}|\nabla \varphi|^4 \right)^{\frac{p}{p-1}}\\
             &=\varepsilon\varphi v^{-2}|\nabla v|^4+\varepsilon\varphi v^{2p}+C(\varepsilon)\left(\tilde{C} M \varphi^{-\frac{3p+1}{4p}}|\nabla \varphi| \right)^{\frac{4p}{p-1}}
    \end{align*}
    It remains to estimate the terms involving \( F_\theta \). From the definition of \( F_\theta \), we find
    $$F_\theta \leq \iint \varphi \left( v_t^2 + \varepsilon v^{2p} + \varepsilon M^2 |\nabla v|^4 + C(\varepsilon) f_\theta^2 \right).$$
    Substituting these estimates together with those from the proof of Lemma \ref{g1} into (\ref{v1}), we have
    \begin{align*}
            &\frac{3}{4}\alpha I+\frac{3}{4}\delta L-\beta G\nonumber\\
            \leq & C(\varepsilon I+\varepsilon L+C(\varepsilon)M^2E)+C\iint\varphi v_t^2+C(\varepsilon)\iint\varphi f_\theta^2\\ &+C(\varepsilon)\iint\left[\left(\varphi^{-\frac{3p+1}{4p}}|\nabla \varphi|\right)^{\frac{4p}{p-1}}+\left(\varphi^{-\frac{p+1}{2p}}|\Delta\varphi|\right)^{\frac{2p}{p-1}}+\left(\varphi^{-\frac{p+1}{2p}}|\varphi_t|\right)^{\frac{2p}{p-1}}\right].
    \end{align*}
    Notice $0<v<\tilde{C}$, then
    $$I=\iint v^{-2}\varphi|\nabla v|^4\geq \tilde{C}^{-2}\iint \varphi |\nabla v|^4=\tilde{C}^{-2}E.$$
    Set $\mu=\frac{\alpha}{4\tilde{C}^2}$, we derive
    \begin{align}\label{f1}
            &\mu E+\frac{\alpha}{2} I+\frac{3}{4}\delta L-\beta G\nonumber\\
            \leq& \frac{3}{4}\alpha I+\frac{3}{4}\delta L-\beta G\nonumber\\
            \leq & C\varepsilon I+C\varepsilon L+C(\varepsilon)M^2E+C\iint\varphi v_t^2+C(\varepsilon)\iint\varphi f_\theta^2\nonumber\\ &+C(\varepsilon)\iint\left[\left(\varphi^{-\frac{3p+1}{4p}}|\nabla \varphi|\right)^{\frac{4p}{p-1}}+\left(\varphi^{-\frac{p+1}{2p}}|\Delta\varphi|\right)^{\frac{2p}{p-1}}+\left(\varphi^{-\frac{p+1}{2p}}|\varphi_t|\right)^{\frac{2p}{p-1}}\right].
    \end{align}
    Finally, we estimate the time derivative term of \( v \). From (\ref{eq2}), we get
    \begin{align}\label{vt}
            \iint& \varphi v_t^2\nonumber\\
            =& \iint \varphi v_t \left( v^p + M |\nabla v|^2 + \Delta v + f_\theta \right)\nonumber\\
            =& \frac{1}{p+1} \iint \varphi \left( v^{p+1} \right)_t + M \iint \varphi v_t |\nabla v|^2 + \iint \varphi v_t \Delta v + \iint \varphi v_t f_\theta\nonumber\\
            \leq& \frac{1}{p+1} \iint |\varphi_t| v^{p+1} + \frac{1}{2} \iint \left( \varphi v_t^2 + M^2 \varphi |\nabla v|^4 \right)\nonumber\\
            &- \iint \left( \nabla \varphi v_t + \varphi \nabla v_t \right) \cdot \nabla v + \frac{1}{4}\iint \varphi v_t^2+\iint\varphi f_\theta^2\nonumber\\
            \leq& C(p) \iint |\varphi_t| v^{p+1} + \frac{3}{4} \iint \varphi v_t^2 + \frac{1}{2}\iint M^2 \varphi |\nabla v|^4 - \iint \varphi \partial_t \left( \frac{|\nabla v|^2}{2} \right)\nonumber\\
            &+ \varepsilon \iint \varphi v^{-2} |\nabla v|^4 + \varepsilon \iint \varphi v^{2p} + C(\varepsilon) \iint \left( \varphi^{-\frac{3p+1}{4p}} |\nabla \varphi| \right)^{\frac{4p}{p-1}} + \iint \varphi f_\theta^2.
    \end{align}
    Substituting (\ref{me}), (\ref{vp}) and (\ref{d1}) into (\ref{vt}), it follows that
    \begin{align}\label{e1}
            \iint \varphi v_t^2 \leq & C \left( \varepsilon I + \varepsilon L + C(\varepsilon) M^2 E\right)\nonumber\\
            &+ C(\varepsilon) \iint \left[ \left( \varphi^{-\frac{3p+1}{4p}} |\nabla \varphi| \right)^{\frac{4p}{p-1}} + \left( \varphi^{-\frac{p+1}{2p}} |\varphi_t| \right)^{\frac{2p}{p-1}} \right] + C \iint \varphi f_\theta^2.
    \end{align}
    Therefore, from (\ref{f1}) and (\ref{e1}), we have
    \begin{align}
            \mu E&+\frac{\alpha}{2} I + \frac{3}{4}\delta L - \beta G\nonumber\\
            \leq & C \left( \varepsilon I + \varepsilon L + C(\varepsilon) M^2 E \right) + C(\varepsilon) \iint \varphi f_\theta^2\nonumber\\
            &+ C(\varepsilon) \iint \left[ \left( \varphi^{-\frac{3p+1}{4p}} |\nabla \varphi| \right)^{\frac{4p}{p-1}} + \left( \varphi^{-\frac{p+1}{2p}} |\Delta \varphi| \right)^{\frac{2p}{p-1}} + \left( \varphi^{-\frac{p+1}{2p}} |\varphi_t| \right)^{\frac{2p}{p-1}} \right].
    \end{align}
    We select $\varepsilon=\varepsilon(N,p)>0$ sufficiently small such that $C\varepsilon<\min\{\frac{\alpha}{4},\frac{\delta}{2}\}$. And there exists $\varepsilon_2(N,p)>0$ such that, for $0<M< \varepsilon_2$,
    $$C(\varepsilon)M^2=C(N,p)M^2\leq\frac{\mu}{2}.$$
    Together with (\ref{n1}), and using the fact $\bar{\alpha}>\frac{3p+1}{4p}$, we obtain
    $$ \frac{\mu}{2} E+\frac{\alpha}{4}I+\frac{\delta}{4} L-\beta G \leq C(N,p)R^{-\frac{4p}{p-1}+N+2}+C(\varepsilon)\iint\varphi f_\theta^2.$$
    Using the fact that $\beta<0$ and
    $$\int_{-R^2/2}^{R^2/2}\int_{B_{R/2}}\varphi v^{2p}(x,t){\rm d}x{\rm d}t\rightarrow \int_{-R^2/2}^{R^2/2}\int_{B_{R/2}}u^{2p}(x,t){\rm d}x{\rm d}t\ \ {\rm as}\ \theta\rightarrow  0,$$
    we derive
    \[\int_{-R^2/2}^{R^2/2} \int_{B_{R/2}} u^{2p}(x, t) \, {\rm d}x \, {\rm d}t \leq C(N, p) R^{-\frac{4p}{p-1} + N + 2} .\]
    Since $p<p_B\leq p_S$, letting $R\to \infty$, we conclude that $u\equiv 0$.
\end{proof}

\section*{Acknowledgments}
The author thanks Prof. Xiaoli Han and Prof. Liang Zhao for their patient guidance and insightful suggestions.

\end{document}